\documentclass[a4paper,12pt]{article}
\usepackage{amssymb}
\usepackage{epsfig}
\usepackage{amsfonts}
\usepackage{amsmath}
\usepackage{euscript}
\usepackage{amscd}
\usepackage{xcolor}
\usepackage{amsthm,relsize}
\usepackage{hyperref}
\DeclareMathAlphabet{\mathpzc}{OT1}{pzc}{m}{it}
\newtheorem{thm}{Theorem}[section]
\newtheorem{lem}[thm]{Lemma}
\newtheorem{prop}[thm]{Proposition} 
\newtheorem{defn}[thm]{Definition}
\newtheorem{cor}[thm]{Corollary}
\newtheorem{rem}[thm]{Remark}
\newtheorem{ex}[thm]{Example}
\newtheorem{ques}[thm]{Question}
\newcommand{\p}{\mathpzc{p}}
\newcommand{\m}{\mathpzc{m}}

\newcommand{\bR}{\mathbb R}
\newcommand{\bQ}{\mathbb Q}
\newcommand{\bN}{\mathbb N}
\newcommand{\bC}{\mathbb C}

\newcommand{\lB}{\mathcal{B}}
\newcommand{\lF}{\mathcal{F}}
\newcommand{\lG}{\mathcal{G}}
\newcommand{\bG}{\mathbb{bG}}
\newcommand{\lH}{\mathcal{H}}
\newcommand{\bH}{\mathbb{H}}

\author{Nikhilesh Dasgupta$^{*}$ and Animesh Lahiri$^{*}$\\
	{\small{\it $^{*}$  MURTI Research Center, GITAM (Deemed to be University),}}\\
	{\small{\it NH 207, Nagadenehalli, Doddaballapura, }}\\
	{\small{\it Karnataka 562163, Bengaluru, India.}}\\
	{\small{\it E-mail: its.nikhilesh@gmail.com, 255alahiri@gmail.com }}}
\title{\bf On image ideals of nice and quasi-nice derivations on $R[X_1,X_2]$}

\begin{document}
\date{}
\maketitle

\abstract{\noindent

	\noindent
	In this paper, we investigate the structure of image ideals associated with irreducible nice and quasi-nice $R$-derivations on $R[X_1, X_2]$, where $R$ is a finitely generated UFD  over a field of characteristic zero. For nice derivations $D$ (i.e., $D^2X_1 = D^2X_2 = 0$), we show that the $j$-th image ideal of $D$ is $(DX_1, DX_2)^j \ker(D)$ for all $j\geqslant 1$. In the more intricate case of strictly $1$-quasi-nice derivations (i.e., $D^2X_1=0$), we provide precise structure for image ideals using weighted degrees and divisibility data from the coefficients of $DX_2$, provided $DX_1$ is an irreducible element of $R$. Moreover, if $DX_1$ and the leading coefficient of $DX_2$ is comaximal in $R$, then the $j$-th image ideal of $D$ is generated by $(DX_1)^j$.  When $R$ is a PID and $DX_1$ factors into primes, we describe the generators of each image ideal in terms of these primes and their relation to the polynomial coefficients of $DX_2$. These results give a unified and computable framework for understanding how derivations shape the algebraic structure of their image ideals.}

\smallskip
\noindent
{\small {{\bf Keywords}. Locally nilpotent derivation, image ideal, polynomial ring, filtration.}
	

	\section{Introduction}
	We will assume all rings to be commutative containing unity. By $k$, we will always denote a field of characteristic zero. The set of all non-negative integers will be denoted by $\bN$. The notation  $R^{[n]}$ will be used to denote an $R$-algebra isomorphic to a polynomial algebra in $n$ variables over $R$. Unless otherwise stated, capital letters like $X_1,\dots,X_n,Z_1,\dots,Z_m,X,Y,Z,U,V,W$ will be used as variables in the polynomial ring. For a ring $R$, by an affine $R$-algebra $B$ we mean that $B$ is a finitely generated algebra over $R$. 
	
	\smallskip
	\noindent
	Let $R$ be a $k$-algebra and $B$ an $R$-algebra. A $k$-linear map $D:B\longrightarrow B$ is said to be a {\it derivation}, if it satisfies the Leibnitz rule: $D(ab)=aDb+bDa,~\forall a,b \in B$. A derivation $D$ is said to be a {\it locally nilpotent derivation} (abbrev. lnd) if, for each $b\in{B}$, there exists $n(b)\in{\bN}$ such that $D^{n(b)}b=0$. The set of all locally nilpotent derivations on $B$ will be denoted by ${\rm LND}(B)$. An $R$-derivation $D:B\longrightarrow B$ is an $R$-linear derivation. The set of all locally nilpotent $R$-derivations (abbrev. $R$-lnd) on $B$ will be denoted by ${\rm LND}_R(B)$. The {\it kernel} of $D$, denoted by ${\rm Ker}(D)$, is defined to be the set $\{b \in B~\vert~Db=0\}$.
	The {\it $j$-th image ideal} of $D$ is the ideal of ${\rm Ker}(D)~(:=A)$, defined by 
	$$ I_j=A \cap D^{j}B\text{ for } j\geqslant 1.$$
	The $1$st image ideal $I_1$ is called the {\it plinth ideal} of $D$. An important problem in the area of locally nilpotent derivations is to study the minimal number of generators of image ideals $I_j$. When $B=k^{[3]}$, it is well-known that the plinth ideal $I_1$ is principal (\cite[Theorem 5.12]{F}, \cite[Theorem 1]{DK}). Freudenburg conjectured that all the image ideals are also principal for $D \in {\rm LND}_k(k^{[3]})$ (\cite[11.2]{F}). This conjecture is called the {\it Freeness Conjecture}.
	
	
	\medskip
	\indent
	In this paper, we consider affine $k$-algebras $R$ which are UFDs and  investigate the image ideals of irreducible nice and quasi-nice $R$-derivations (see Definition \ref{niceqnice}) on $R^{[2]}$. Nice and quasi-nice derivations form a natural intermediate class between
general locally nilpotent derivations and the more restrictive \emph{triangular}
derivations. Recall that a derivation $D$ on $B = R^{[n]}$ is triangular if,
in some coordinate system, $DX_1\in R$ and  $DX_i \in R[X_1, \ldots, X_{i-1}]$ for each $i\geqslant 2$.
By a classical theorem of Rentschler (\cite{R}), every $R$-lnd on $R[X,Y]$ is
triangularizable when $R$ is a field, but this fails in dimension
three. The condition $D^2 X_i = 0$ defining a nice (or quasi-nice)
derivation is precisely says that $DX_i$ lie in $\mathrm{Ker}(D)$,
which is considerably weaker than  triangularity yet still strong
enough to describe the structure of ${\rm Ker}(D)$. Such
derivations arise naturally in the study of $\mathbb{G}_a$-actions on affine
spaces that fail to be triangularizable or even locally trivial: for
instance, the Deveney--Finston example of a proper, non-locally-trivial
$\mathbb{G}_a$-action on $\mathbb{C}^5$ (\cite{DeF}) is built
from a derivation of this type.

\smallskip\noindent The kernel of an irreducible $R$-lnd on $R[X,Y]$ was completely described
by Bhatwadekar and Dutta when $R$ is a Noetherian domain containing $k$
(\cite{BD97}, see Theorem~\ref{BhD} below), and this
description was sharpened by Wang (\cite{W}) in the nice and quasi-nice
case when $R$ is a UFD. In fact he showed that $\mathrm{Ker}(D)$ is again a polynomial ring
$R^{[1]}$, generated by a linear polynomial in the nice case or  by a polynomial which is linear in one of the variables in the quasi-nice case (see Lemma~\ref{Wang}). In three
variables, the first author and Gupta (\cite{DG}) showed that the
kernel of an irreducible nice $R$-lnd on $R[X,Y,Z]$, for $R$ a PID, is a
polynomial ring $R^{[2]}$ containing a coordinate in $R[X,Y,Z]$ (see Theorem~\ref{Nice}). These
results give a fairly complete picture of the kernels of nice and
quasi-nice derivations, but say little about finer ideal-theoretic
invariants such as the image ideals $I_j$, which are studied in the present
paper.

\medskip\noindent
For any $k$-domain $B$ and an lnd $D$ having a slice (i.e., $\exists~s\in B$ such that $Ds=1$), it is easy to observe (see Lemma \ref{ii_slice}) that all the image ideals are equal to the whole ring and hence principal. In fact, if $R$ is an affine $k$-domain, then any fixed point free $R$-lnd $D$ (i.e., $(DB)B=B$) on $B=R^{[2]}$  has a slice in $B$ and hence all the image ideals are principal (Corollary \ref{2varufd}). So, it is interesting to investigate the situation when $D$ is not fixed point free. The main results proved in this paper are Theorem \ref{inice}, Theorem \ref{2varquasi} and Theorem \ref{plinth_quasi_PID}.
	
	\medskip\noindent
	{\bf Theorem \ref{inice}} {\it Let $R$ be a {\rm UFD} and $B=R[X_1,X_2]$. Let $D$ be an irreducible $R$-lnd on $B$ such that $D^2X_1=D^2X_2=0$. Let $A={\rm Ker}(D)$. Then, for each $j\in \bN, I_j ={( DX_1,DX_2)}^jA$.}
	
	\medskip\noindent
	{\bf Theorem \ref{2varquasi}.} {\it Let $R$ be a {\rm UFD} and $B=R[X_1,X_2]$. Let $D$ be an irreducible $R$-lnd on $B$ which is not fixed point free. Let $A={\rm Ker}(D)$. Suppose that the following conditions hold.}
	\begin{enumerate}
\item[\rm(i)]$DX_1~(=b)\in R$ is irreducible,
\item[\rm(ii)]$DX_2=-f^{\prime}(X_1)$ where $f(X_1)=\overset{m}{\underset{i=0}{\sum}} f_i{X_1}^i$, $f_m\neq 0$,
\item[\rm(iii)]$D$ is strictly $1$-quasi-nice. 
\end{enumerate}
Let $d={\rm Max}\{i\in\bN:i\leqslant m \text{ and } b\nmid f_i\}$. Then, $d\geqslant 2$ and for each $j\geqslant 2$, $I_j=b^{m_j}\bar{I_j}$, where $m_j={\rm Min}\{i_1+(d-1)i_2:i_1,i_2\in\bN,i_1+di_2=j\}$, $\bar{I_j}={(b,f_d)}^{l_j}A$ and $l_j=[j/d].$

\medskip\noindent
Recently, Khaddah, Kahoui and Ouali showed that if $R$ is a PID containing $k$, $B=R^{[2]}$ and $D\in{\rm LND}_{R}(B)$, then all the image ideals of $D$ are principal (\cite[Theorem 4.1]{KKO22}) and as a consequence the freeness conjecture of Freudenburg is true for locally nilpotent derivations on $k^{[3]}$ having atleast one coordinate in ${\rm Ker}(D)$. However, there has not been much study about how a generator for each of the image ideals would look like. For a PID $R$ and an irreducible quasi-nice $R$-lnd $D$ on $R^{[2]}$, we have the following result.

%
%

\medskip\noindent
{\bf Theorem \ref{plinth_quasi_PID}.} {\it Let $R$ be a {\rm PID}, $B=R[X_1,X_2]$ and $D$ an irreducible $R$-lnd on $B$ satisfying the following conditions.

\begin{enumerate}
\item [\rm(i)] $DX_1=\underset{i\in I}\prod{p_i}^{r_i}~(\in R)$, where $I=\{1,2,\dots,n\}$ and for each $i\in I$, $p_i$ is a prime element in $R$.
\item[\rm(ii)] $DX_2=-f^{\prime}(X_1)$, where $f(X_1)=\overset{d}{\underset{j=0}{\sum}} f_jX^j, f_d\neq 0$,
\end{enumerate}
Let $A={\rm Ker}(D)$ and  $J=\{i\in I:p_i\mid f_j \text{ for all }j=2,\dots,d \text{ and }p_i\nmid f_1\}$. Then the following hold.

\begin{enumerate}
\item[\rm(I)] $I_1=
\begin{cases}
	({\underset{i\in I\setminus J}\prod}{p_i}^{r_i})A, & \text{if } J\neq I;\\
	A, & \text{if } J=I.
\end{cases}
$
\item[\rm(II)]  If for each $i\in I$, $p_i\nmid f_d$, then for each integer $j\geqslant 1$, $I_j=(DX_1)^{m_j}A$, where $m_j={\rm Min}\{i_1+(d-1)i_2:i_1,i_2\in\bN,i_1+di_2=j\}$. 

\end{enumerate}}

\medskip\indent
Corollary \ref{primality} and Proposition \ref{top degree ideal} are the essential tools used to study the higher image ideals in case of both nice and quasi-nice derivations. Corollary \ref{primality} gives an explicit set of generators for each $I_n$ under the primality of an ideal $\widetilde{J}$, and Proposition \ref{top degree ideal} is used to find the generators of $\widetilde{J}$. In fact, given an ideal $I$ in a polynomial ring over an affine $k$-domain, equipped with a weighted degree map, using Proposition \ref{top degree ideal} one can compute the ideal generated by the top degree terms of elements of $I$, provided at most one of the generators of $I$ is non-homogeneous. The proof of Corollary \ref{primality} uses some techniques of LND-filtration introduced by B. Alhajjar in \cite{A}.

\section{Preliminaries}

First, we will recall some useful definitions.
\begin{defn}\label{irreducible lnd}
{\em Let $R$ be a $k$-algebra and $B$ an $R$-algebra. An $R$-derivation $D$ on $B$ is said to be {\it irreducible} if there does not exist any non-unit $b$ in $B$ such that $DB\subseteq bB$.}
\end{defn}
\begin{defn}\label{niceqnice}
{\em Let $R$ be a $k$-domain, $B=R^{[n]}$ and $m ~(\leqslant n)$ be a positive integer. An $R$-lnd $D$ on $B$ is said to be {\it m-quasi-nice} (or simply {\it quasi-nice}) if there exists a coordinate system $X_1,X_2, \dots ,X_n$ in $B$ such that $D^{2}(X_i)=0$ for all 
	$i \in \{1,\dots,m\}$. If $m=n$, then $D$ is called 
	{\it nice}.}
	\end{defn}
	
	\begin{defn}\label{strictquasi}
{\em An $m$-quasi-nice derivation is said to be {\it strictly $m$-quasi-nice} if it is not $r$-quasi-nice for any positive 
	integer $r > m$. 
}
\end{defn}

\begin{defn}
\emph {Let $B$ be a ring. A family of additive subgroups $\{\mathcal{B}_i\}_{i\in\bN}$ of $B$ is said to be an {\it $\bN$-filtration} of $B$ if the following conditions hold.
	\begin{enumerate}
		\item[\rm(i)] $\lB_i\subseteq\lB_{i+1}$ for each $i\in\bN$,  
		\item[\rm(ii)] $\underset{{i\in\bN}}\bigcup\mathcal{B}_i=B$,
		\item[\rm(iii)] $\mathcal{B}_i\mathcal{B}_j\subseteq{\mathcal{B}_{i+j}}$ for all $i,j\in\bN$.
	\end{enumerate}
	\noindent
	An $\bN$-filtration $\{\mathcal{B}_i\}_{i\in\bN}$ of $B$ is said to be {\it proper} if the following additional conditions hold.
	\begin{enumerate}
		\item[\rm(iv)]$\bigcap_{i\in\bN}\lB_i=\{0\}$,
		\item[\rm(v)]$a\in\lB_i\setminus\lB_{i-1}$, $b\in\lB_j\setminus\lB_{j-1}$ will imply $ab\in\lB_{i+j}\setminus\lB_{i+j-1}$.
\end{enumerate}}
\end{defn}

\begin{defn}

\emph {Let $B$ be a ring. A map $\theta:B\longrightarrow \bN\cup\{-\infty\}$ is said to be an {\it $\bN$-semi-degree map} on $B$ if the following conditions hold.
	\begin{enumerate}
		\item[\rm(i)] $\theta(a)=-\infty$ if and only if $a=0$,
		\item[\rm(ii)] $\theta(ab)\leqslant\theta(a)+\theta(b)$ for all $a,b\in B,$
		\item[\rm(iii)] $\theta(a+b)\leqslant{\rm Max}~\{\theta(a),\theta(b)\}$ for all $a,b\in B.$
	\end{enumerate}
	\noindent
	An $\bN$-semi-degree map $\theta$ is called a {\it degree map} if in condition (ii) equality occurs for all $a,b\in B$.}

\end{defn}
\noindent
There is a one-one correspondence between proper $\bN$-filtrations of $B$ and $\bN$-degree maps on $B$. In fact, if $\lB:=\{\lB_i\}_{i\in\bN}$ is a proper $\bN$-filtration of $B$, then it induces a degree map $\theta_{\lB}$ on $B$ defined by $\theta_\lB(a)=
\begin{cases}
-\infty,& \text{if }a=0,\\
n,& \text{if }a\in{\lB_n\setminus\lB_{n-1}}.
\end{cases}
$ \\Conversely, if $\theta$ is a degree map on $B$, then $B=\underset{i\in\bN}\bigcup \lB_i$, where $\lB_i=\{a\in B:\theta(a)\leqslant i\}$. It is easy to check that $\{\lB_i\}_{i\in\bN}$ is a proper $\bN$-filtration of $B$.

\begin{defn}
\emph {Let $R$ be a ring and $B=R^{[n]}$. A real-valued degree map $\theta$ on $B$ is said to be a {\it weighted degree map} on $B$, if for any $p\in B$, $$\theta(p)={\rm Max}\{\theta(m):m\in M(p)\},$$ where $M(p)$ is the set of all monomials occurring in the expression of $p$.
}

\smallskip\noindent
\emph{If $M'(p):=\{m\in M(p):\theta(m)=\theta(p)\}$, 
	then we will denote $\sum \limits_{m\in{M'(P)}} m$ by by $\widetilde{p}$. For an ideal $I$ of $B$, $\widetilde{I}$ will denote the ideal of $B$ generated by $\{\widetilde{p}:p\in I\}$.
}
\end{defn}

\begin{defn}
\emph{Let $R$ be a $k$-algebra, $B$ an $R$-algebra, $D\in{\rm LND}_{R}(B)$. For $f\in B$, we define  ${\rm deg}_D(f):=\begin{cases}
-\infty, & \text{ if }f=0;\\
   {\rm Min}\{n\in \bN:D^{n+1}(f)=0\}, & \text{ if }f\neq 0. 
\end{cases}
$}
\end{defn}

\smallskip
Now, we will record some elementary facts and some important results  which will be used throughout the rest of the paper. First, we state and prove an elementary result which follows from the higher product rule for a derivation (\cite[Proposition 1.6]{F}).
\begin{lem}\label{prule}
Let $R$ be a $k$-algebra, $B$ an $R$-algebra, $D\in{\rm LND}_{R}(B)$ and $f_1,\dots,f_n\in B$ with ${\rm deg}_D(f_i)=m_i$. If $m=\sum_{i=1}^{n}m_i$, then we have the following:
\begin{enumerate}
	\item[\rm(i)] $D^m(f_1\dots f_n)=\dfrac{m!}{\prod_{i=1}^{n}m_i!}\prod_{i=1}^{n}D^{m_i}f_i$,
	\item[\rm(ii)]$D^{m+t}(f_1\dots f_n)=0$, where $t\geqslant 1$. In particular, ${\rm deg}_D(f_1\dots f_n)=m$.
\end{enumerate}
\end{lem}
\begin{proof} Proof of (i) follows from the higher product rule, $D^m(fg)=\sum\limits_{i+j=m} {m\choose i}D^{i}fD^{j}g$. The proof of (ii) follows from the fact that for each $i \in \{1,\dots ,n \}$, $D^lf_i=0$ if and only if $l > m_i$.
\end{proof}

\smallskip
\indent
Next, we state an observation describing the image ideals of a locally nilpotent derivation with a slice.
\begin{lem}\label{ii_slice}
Let $R$ be a $k$-algebra, $B$ an $R$-algebra, $D \in{\rm LND}_R(B)$ and $A={\rm Ker}(D)$. Then the following are equivalent.
\begin{enumerate}
	\item[\rm(i)]$D$ has a slice.
	\item[\rm(ii)]$I_1=A$.
	\item[\rm(iii)]$I_n=A$ for all $n \in \bN$.
	\item[\rm(iv)] $I_n=A$ for some $n\in\bN$.
\end{enumerate} 
\end{lem}

\begin{proof} $I_1=A \Rightarrow D$ has a slice (say $s$) $\Rightarrow I_n=A ~\forall n \in \bN$  (since $1=D^{n}(\frac{s^n}{n!})\in A \cap D^nB=I_n$) $\Rightarrow I_n=A$ for some $n\in\bN \Rightarrow I_1=A$.

%
\end{proof}

\indent
For a Noetherian $k$-domain $R$ and an irreducible $D \in{\rm LND}_R(R[X,Y])$, the structure of ${\rm Ker}(D)$   is given in \cite[Theorem 4.7]{BD97}.

\begin{thm}\label{BhD}
Let $R$ be a Noetherian domain containing $k$ and $B=R[X,Y]$. Let $D$ be an irreducible $R$-lnd on $B$ and $A={\rm Ker}(D)$. Then $A=R^{[1]}$ if and only if one of the following conditions hold.
\begin{enumerate}
	\item [\rm(i)] $DX$ and $DY$ form a regular sequence in $B$.
	\item [\rm(ii)] $(DX,DY)B=B$.
\end{enumerate}
Moreover, if {\rm(ii)} holds, then $B=A^{[1]}$.
\end{thm}

\noindent
As an immediate corollary to Theorem \ref{BhD} we have the following result.
\begin{cor}\label{2varufd}
Let $R$ be a Noetherian domain containing $k$, $B=R^{[2]}$ and $D \in {\rm LND}_{R}(B)$ be fixed point free. Then $I_n=A$ for all $n \in \bN$.
\end{cor}

\begin{proof} The proof follows from Lemma \ref{ii_slice} and Theorem \ref{BhD}.
\end{proof}

\smallskip
\indent 
The following result of Z. Wang (\cite[Lemma 4.1 and Lemma 4.2]{W}) describes ${\rm Ker}(D)$ for an irreducible nice (or quasi-nice) $R$-lnd $D$ on $R[X,Y]$, when $R$ is a UFD containing $k$.  
\begin{lem}\label{Wang}
Let $R$ be a {\rm UFD} containing $k$, $B=R[X,Y]$ and $D$ an irreducible $R$-lnd. Then, the following hold:
\begin{enumerate}
	\item [\rm (i)] If $D^{2}X=0$, then ${\rm Ker}(D)=R[bY+f(X)]=R^{[1]}$, where $b \in R$ and $f(X) \in R[X]$. Moreover, $DX \in R$ and $DY \in R[X]$.
	\item [\rm (ii)] If $D^{2}X=D^{2}Y=0$, then $D=b\dfrac{\partial }{\partial X}-a\dfrac{\partial }{\partial Y}$ for some $a,b \in R$. Moreover, ${\rm Ker}(D)=R[aX+bY]$.
	\item [\rm (iii)] If $R$ is a {\rm PID} and $D^{2}X=D^{2}Y=0$, then $D$ has a slice.
\end{enumerate}
\end{lem}

\smallskip
\indent
The following result (\cite[Theorem 3.6]{DG}) describes the structure of the kernels of nice derivations on $R^{[3]}$, where $R$ is a PID containing $k$ .

\begin{thm}\label{Nice}
Let $R$ be a {\rm PID} containing $k$ and $B=R[X,Y,Z]$.
Let $D$ be an irreducible $R$-lnd on $B$ with $D^{2}X=D^{2}Y=D^{2}Z=0$. Let $A={\rm Ker}(D)$. 
Then there exists a coordinate system $(U,V,W)$ in $B$ related to $(X,Y,Z)$ by a linear change such that the following hold:
\begin{enumerate}
	\item [\rm (i)] $A$ contains a non-zero linear form of $\{X,Y,Z\}$.
	\item [\rm (ii)] ${\rm rank}(D)\leq 2$. In particular, $A=R^{[2]}$.
	\item [\rm (iii)] $A=R[U,gV-fW]$, where $DV=f$, $DW=g$ and $f,g\in R[U]$ be such that ${\rm gcd}_{R[U]}(f,g)=1$.
	\item [\rm (iv)] Either $f$ and $g$ are comaximal in $B$ or they form a regular sequence in $B$. Moreover, if they are 
	comaximal, then $B=A^{[1]}$ and ${\rm rank}(D)=1$; and if they form a regular sequence, then $B$ is not $A$-flat and ${\rm rank}(D)=2$. 
\end{enumerate}
\end{thm}

\smallskip
\indent
Next, we state a generalization of a result proved by S. Kaliman and L. Makar-Limanov (\cite{KML}). The proof given by the authors in \cite{KML} for $R=\bC$ works for any affine $k$-domain $R$ with a slight modification of the hypotheses.
\begin{lem}\label{MKL}
Let $R$ be an affine $k$-domain, $C=R[X_1,\dots,X_n]$ and $\lambda$ a weighted degree map on $C$. Suppose  $J$ is an ideal of $C$ contained in the ideal $(X_1,\dots,X_n)C$ and \\$\pi:C \rightarrow C/J$ is the canonical map. Define $\eta:C/J\longrightarrow\bR\cup\{-\infty\}$ by $$\eta(\alpha)=
\begin{cases}
	{\rm inf}\{\lambda(h):h\in{\pi}^{-1}(\alpha)\},& \text{if } \alpha\neq 0, \\
	-\infty,& \text{if }\alpha=0,
\end{cases}
$$ Then, for $\alpha\neq 0$, the following hold:
\begin{enumerate}
	\item[\rm(i)] there exists $h\in\pi^{-1}(\alpha)$ such that $\widetilde{h}\notin{\widetilde{J}}$,
	\item[\rm(ii)]$\eta(\alpha)=\lambda(h)$ for some $h\in\pi^{-1}(\alpha)$ if and only if $\widetilde{h}\notin\widetilde{J}$,
	\item[\rm(iii)] $\eta$ is a semi-degree map on $C/J$. If $\widetilde{J}$ is a prime ideal of $C$, then $\eta$ is a degree map.
\end{enumerate}
\end{lem}

\section{LND filtrations and Top degree ideal}
In this section we prove some technical results, which are crucial for our main results in Sections $4$ and $5$. Throughout the rest of the section, $R$ will denote an affine $k$-domain. Let $B=R[X_1,\dots,X_n]$. 
Let $D\in{\rm LND}_R(B)\setminus\{0\}$. Suppose $A={\rm Ker}(D)$ is a finitely generated $R$-algebra. Then $I_1=(Ds_1,\dots,Ds_m)A$, for some $s_1,\dots,s_m\in B$. Without loss of generality, we can assume that $s_i(0,\dots,0)=0$, for each $i$.

\smallskip\noindent
Consider the filtration $\{\mathcal{F}_j\}_{j\in{\bN }}$ defined by
$\mathcal{F}_j={\rm Ker}(D^{j+1})$. Then it is easy to see the following:
\begin{enumerate}
\item[\rm(i)] $\mathcal{F}_0=A$,
\item[\rm(ii)] $\mathcal{F}_1=\sum_{l=1}^{m}s_l\mathcal{F}_0+\mathcal{F}_0$,
\item[\rm(iii)] $I_j=\big{(}D^j(\mathcal {F}_j)\big{)} A$ for each $j\in\bN$. 

\end{enumerate}

\noindent 
Let $u_l:={\rm deg}_D(X_l)$ for $l=1,\dots,n.$ For each $j\in\bN$, define $\mathcal{G}_j$ in the following way:
$$\mathcal{G}_j=\underset{\underset{j_1,\dots,j_{m+n}\geqslant 0}{j_1+\dots+j_m+u_1j_{m+1}+\dots+u_nj_{m+n}=j}}{\text{\Large $\sum$}}\big{(}s_1^{j_1}\dots s_m^{j_m}X_1^{j_{m+1}}\dots X_n^{j_{m+n}}\big{)}\lF_0.$$

\noindent
We note the following easy observations:
\begin{enumerate}
\item[\rm(i)] $\lG_0=\lF_0$,
\item[\rm(ii)] $\lG_i\lG_j\subseteq\lG_{i+j}$, for all $i,j\in\bN$.
\end{enumerate}

\noindent
For each $j\in\bN$, let
$\mathcal{H}_j=\sum_{i=0}^j\mathcal{G}_i$. Clearly, for each $j\in\bN$, $\lH_j\subseteq\lH_{j+1}$. For $i,j,q,r\in\bN$ with $0\leqslant q\leqslant i$ and $0\leqslant r\leqslant j$, since $\lG_q\lG_r\subseteq\lG_{q+r}\subseteq{\lH_{i+j}}$, we have $\lH_i\lH_j\subseteq\lH_{i+j}$.
Let $\alpha=\sum a_{i_1,\dots,i_n}{X_1}^{i_1}\dots{X_n}^{i_n}\in{B}$ and $S=\{(i_1,\dots,i_n)\in{\bN}^n:a_{i_1,\dots,i_n}\neq 0\}$.\\ If $n_0={\rm Max}~\{u_1i_1+\dots+u_ni_n:(i_1,\dots,i_n)\in S\}$, then $\alpha \in\lH_{n_0}$. So, $B=\bigcup_{j\in\bN}\lH_j$. Hence $\{\lH_j\}_{j\in\bN}$ is an $\bN$-filtration of $B$. When $\{\lH_j\}_{j\in\bN}$ is a proper $\bN$-filtration, the following lemma gives a generating set for each $I_j$.
\begin{lem}\label{jth image ideal}
If $\{\lH_j\}_{j\in\bN}$ is a proper $\bN$-filtration, then $I_j=\big{(} D^j(\mathcal{G}_j)\big{)}A$ for all $j\in{\bN}$.
\end{lem}

\begin{proof}
By Lemma \ref{prule}, $\lH_j\subseteq\lF_j$ for each $j\in\bN$. It is clear that $\lH_0=\lG_0=\lF_0$ and $\lF_1~(=\sum_{l=1}^ms_l\lF_0+\lF_0)\subseteq{\lG_1}+\lG_0=\lH_1$. Hence $\lF_j=\lH_j$ for $j=0,1$. Let $\alpha\in\lF_j$ for some $j\in\bN~(j\geqslant 2)$ and $\theta$ be the degree map on $B$ corresponding to the filtration $\{\lH_j\}_{j\in\bN}$ of $B$.
Let $s\in{\lF_1\setminus\lF_0}$ and $f=Ds$. By Principle $11$ in \cite{F}, $B[1/f]=A[1/f][s]$ and hence $\exists~\epsilon\in\bN$ and $a_0,\dots a_j\in A$ such that $f^\epsilon\alpha=a_0+a_1s+\dots+a_js^j$. Since $\lF_j=\lH_j$ for $j=0,1$, we have $\theta(s)=1$ and $\theta(f)=0=\theta(a_i)$ for all $i=0,\dots,j$. Hence $\theta(\alpha)\leqslant j$. So, $\alpha\in\lH_j$. Thus $\lH_j=\lF_j$ for each $j\in\bN$. Since $\lH_j=\lG_j+\lH_{j-1}$, we have $\lF_j=\lG_j+\lF_{j-1}$ for all $j\in\bN$. Hence $I_j=\big{(}D^j(\lG_{j})\big{)}A$ for all $j\in\bN$.
\end{proof}

\indent Now, we will consider the situation when ${\rm Ker}(D)=R[z_1,\dots,z_r]$ for some $z_1,\dots,z_r$ in $R[X_1,\dots,X_n]$ with $z_i(0,\dots,0)=0$ for all $i$, $1\leqslant i\leqslant r$. We define a weighted $\bN$-degree map $\lambda$ on $C:=R[X_1,\dots,X_n,Z_1,\dots,Z_r,S_1,\dots,S_m]~(=R^{[n+r+m]})$, by setting $\lambda(r)=0$ for all $r\in R$, $\lambda(X_l)=u_l$ for  $1\leqslant l\leqslant n$, $\lambda(Z_l)=0$ for $1\leqslant l\leqslant r$ and $\lambda(S_l)=1$ for $1\leqslant l\leqslant m$. Let $J$ be the ideal of $C$ generated by the set $\{Z_1-z_1,\dots,Z_r-z_r,S_1-s_1,\dots,S_m-s_m\}$. Then $C/J\cong B$. By Lemma \ref{MKL}, $\lambda$ induces a semi-degree map $\eta$ on $B$ defined for each $\alpha\in B\setminus\{0\}$, $$\eta(\alpha)={\rm inf}\{\lambda(h):h\in{\pi}^{-1}(\alpha)\}, \text{where }\pi:C \rightarrow B \text{ is the canonical map}.$$

\smallskip\noindent
The next result shows that the filtration induced by $\eta$ is $\{\mathcal{H}_j\}_{j\in\bN}$.

\begin{lem}\label{induced fil of eta}
$\{\lH_i\}_{i\in\bN}$ is the filtration induced by the semi-degree map $\eta$ on $B$.
\end{lem}
\begin{proof}
Let $\{\bH_i\}_{i\in\bN}$ be the filtration of $C$ induced by the weighted degree map $\lambda$ on $C$. Then, $\bH_i=\text{\Large$\bigoplus$}_{j=0}^i{\bG}_j,$ where, for each $j$ with $0\leqslant j\leqslant i$, $${\bG}_j=\underset{\underset{j_1,\dots,j_{m+n}\geqslant 0}{j_1+\dots+j_m+u_1j_{m+1}+\dots+u_nj_{m+n}=j}}{\text{\Large$\bigoplus$}}\big{(}S_1^{j_1}\dots S_m^{j_m}X_1^{j_{m+1}}\dots X_n^{j_{m+n}}\big{)}R[Z_1,\dots,Z_r].$$ Clearly, for each $j$ with $0\leqslant j\leqslant i$, we have $\pi(\bG_j)=\lG_j$ and hence for each $i\in\bN$, $\pi(\bH_i)=\lH_i$.

\smallskip\noindent
Let $\alpha\in B\setminus\{0\}$ be such that $\eta(\alpha)\leqslant i$. By part (i) and (ii) of Lemma \ref{MKL}, there exists $h\in\pi^{-1}(\alpha)$ such that $\lambda(h)=\eta(\alpha)$ and hence $h\in\bH_i$. So, $\alpha~\big{(}=\pi(h)\big{)}\in\pi(\bH_i)=\lH_i$.

\smallskip\noindent
Since $\eta(X_l)\leqslant\lambda(X_l)=u_l$ for all $l\in\{1,\dots,n\}$, $\eta(s_l)\leqslant \lambda(S_l)=1$ for all $l\in\{1,\dots,m\}$, $\eta(z_l)\leqslant \lambda(Z_l)=0$ for all $l\in\{1,\dots,r\}$ and $\eta$ is a semi-degree map on $C$, we have $\eta(\alpha)\leqslant i$ for all $\alpha\in\lH_i$.
\end{proof}

\noindent
From Lemma \ref{MKL}(iii), Lemma \ref{induced fil of eta}  and Lemma \ref{jth image ideal}, we immediately deduce the following corollary.
\begin{cor}\label{primality}
If $\widetilde{J}$ is a prime ideal of $C$, then for each $j\geqslant 0$ we have $I_j=\big{(}D^j(\lG_j)\big{)} A$.
\end{cor}

\indent
Suppose $\lambda$ is a weighted degree map on $B~(=R^{[n]})$ and $J=(f_1,\dots,f_m)$ is an ideal of $B$. It is interesting to investigate conditions under which $\widetilde{J}=(\widetilde{f_1},\dots,\widetilde{f_m})$. If $m=1$, then it is easy to see that $\widetilde{J}=(\widetilde{f_1})$. But for $m\geqslant 2$ this is not the case. The following proposition gives sufficient conditions under which $\widetilde{J}=(\widetilde{f_1},\dots,\widetilde{f_m})$.
\begin{prop}\label{top degree ideal}
Let $\lambda$ be a weighted degree map on $B~(=R^{[n]})$ and $J=(f_1,\dots,f_m)$ an ideal of $B$. Suppose the following conditions hold:
\begin{enumerate}
	\item [\rm(i)]$\widetilde{f_i}=f_i$ for $1\leqslant i\leqslant m-1$,
	\item[\rm(ii)] $\widetilde{f_m}$ is a non-zerodivisor in $B/(f_1,\dots,f_{m-1})$.
\end{enumerate}

\noindent
Then $\widetilde{J}=(f_1,\dots,f_{m-1}, \widetilde{f_m})$.
\end{prop}
\begin{proof}
Since for each $i$, $1\leqslant i\leqslant m$, $\widetilde{f_i}\in\widetilde{J}$, we clearly have $(f_1,\dots,f_{m-1}, \widetilde{f_m})\subseteq \widetilde{J}$.

\medskip\noindent
Let $I=\{1,\dots,m\}$ and $0\neq g={\underset{i\in I}\sum}r_if_i$, where $r_i\in B$ for all $i\in I$. Let $J_m:=(f_1,\dots,f_{m-1})$ and for each $i\in I$, $\lambda(r_i)=u_i$ and $\lambda(f_i)=v_i$. For $f\in B$, by $f^{(l)}$ we will denote the $l$-degree homogeneous component of $f$ and for $0\leqslant s_i\leqslant u_i$, $0\leqslant t_i\leqslant v_i$, we define $x(i;s_i;t_i):=r_i^{(u_i-s_i)}f_i^{(v_i-t_i)}$. Then $g={\underset{0\leqslant t_i\leqslant v_i}{\underset{0\leqslant s_i\leqslant u_i}{\underset{i\in I}\sum}}}x(i;s_i;t_i).$

\smallskip\noindent Let $M=\{x(i;s_i;t_i):i\in I,0\leqslant s_i\leqslant u_i,0\leqslant t_i\leqslant v_i\}$, $\mu_0:={\rm Max}\{\lambda(x):x\in M \}$ and $M_0:=\{x\in M:\lambda(x)=\mu_0\}$. Clearly, if $x(i;s_i;t_i)\in M_0$, then $s_i=0=t_i$. 

\smallskip\noindent If ${\underset{x\in{M_0}}\sum}x\neq 0$, then $\widetilde {g}={\underset{x\in{M_0}}\sum}x={\underset{M_0}\sum}\widetilde{r_i}\widetilde{f_i}\in (J_m,\widetilde{f_m}).$ If ${\underset{x\in{M_0}}\sum}x=0$, then we will define $\mu_1$ and $M_1$ in the following way:
$$\mu_1:={\rm Max}\{\lambda(x):x\in M\setminus M_0 \} \text{ and }M_1:=\{x\in M\setminus M_0:\lambda(x)=\mu_1\}.$$ 
If ${\underset{x\in{M_1}}\sum}x\neq 0$, then $\widetilde {g}={\underset{x\in{M_1}}\sum}x$. In this situation we note that  for some $t_i>0$, $x(i;s_i;t_i)\in M_1$ will imply $i=m$. Clearly, $x(m;s_m;t_m)\in M_1$ for some $t_m>0$ only if $x(m;s_m,0)\in M_0$ and hence $s_m=0$. Since ${\underset{x\in{M_0}}\sum}x= 0$, by hypothesis (ii) of the proposition, $x(m;0;0)\in M_0$ only if $r_m^{(u_m)}\in J_m$. So, for some $t_m>0$, $x(m;s_m;t_m)\in M_1$ only if $x(m;s_m;t_m)~\big{(}=x(m;0;t_m)=r_m^{(u_m)}f_m^{(v_m-t_m)}\big{)}\in J_m$. Hence, ${\underset{x\in{M_1}}\sum}x\neq 0$ will imply $\widetilde{g}\in (J_m,\widetilde{f_m})$.

\medskip\noindent
If ${\underset{M_1}\sum}x=0$, then we will define $\mu_2$ and $M_2$ in the following way and proceed similarly.
$$\mu_2:={\rm Max}\{\lambda(x):x\in (M\setminus M_0)\setminus M_1 \}$$ $$\text{and }M_2:=\{x\in (M\setminus M_0)\setminus M_1:\lambda(x)=\mu_2\}.$$ 
Now, for each positive integer $N$ we consider the following statement:
$$P(N):\text{ If }{\underset{x\in M_j}\sum}x=0 \text{ for all } j<N, \text{ then, for }t_m>0, x(m;s_m;t_m)\in M_N\text{ only if }x(m;s_m;t_m)\in J_m.$$

\noindent
We have already observed that $P(1)$ is true. Assume that $P(1),P(2),\dots,P(N)$ is true for some $N\geqslant 1$. We will show that $P(N+1)$ is also true. Let ${\underset{x\in M_j}\sum}x=0$ for all $j\leqslant N$ and $x(m;s_m^{'};t_m^{'})\in M_{N+1}$ for some fixed $s_m^{'},t_m^{'}$ with $0\leqslant s_m^{'}\leqslant u_m$ and $1\leqslant t_m^{'}\leqslant v_m$. Then $x(m;s_m^{'};0)\in M_{j_0}$ for some $0\leqslant j_0\leqslant N$. We observe the following:
\begin{enumerate}
	\item $0={\underset{M_{j_0}}\sum}x(i;s_i;t_i)={\underset{i\neq m}{\underset{M_{j_0}}\sum}}x(i;s_i;0)+{\underset{M_{j_0}}\sum}x(m;s_m;0)+{\underset{t_m>0}{\underset{M_{j_0}}\sum}}x(m;s_m;t_m)$,
	\item ${\underset{i\neq m}{\underset{M_{j_0}}\sum}}x(i;s_i;0)\in J_m$,
	\item Since $P(j_0)$ is true, ${\underset{t_m> 0}{\underset{M_{j_0}}\sum}}x(m;s_m;t_m)\in J_m$,
	\item Since $x(m;s_m^{'};0)\in M_{j_0}$, $x(m;s_m;0)\in M_{j_0}$ if and only if $s_m=s_m^{'}$.
\end{enumerate}
\medskip\noindent 
Hence $x(m,s_m^{'},0)\in J_m$ and  by hypothesis (ii), $r_m^{(u_m-s_m^{'})}\in J_m$. So, $x(m;s_m^{'};t_m^{'})\in J_m$. Thus, $P(n+1)$ is true and hence $P(n)$ is true for all $n\in\bN$.

\medskip\noindent
Now, for $n\in\bN$, if ${\underset{x\in M_j}\sum}x= 0$ for all $j<n$ and ${\underset{x\in M_n}\sum}x\neq 0$, then $\widetilde{g}={\underset{x\in M_n}\sum}x\in(J_m,\widetilde{f_m})$. If ${\underset{x\in M_j}\sum}x=0$ for all $j\leqslant n$, then we will consider $\mu_{n+1},M_{n+1}$ and repeat the arguments for $M_{n+1}$. Since $M$ is a finite set this process will stop after a finite number of steps.
\end{proof}

\section{Structure of $I_n~(n\geqslant 1)$ for nice derivations}
Let $R$ be a UFD, finitely generated over $k$. In this section, we compute the structure of all the image ideals of an irreducible nice $R$-derivation on $R^{[2]}$.
\begin{thm}\label{inice}
Let $R$ be a {\rm UFD} and $B=R[X_1,X_2]$. Let $D$ be an irreducible $R$-lnd on $B$ such that $D^2X_1=D^2X_2=0$. Let $A={\rm Ker}(D)$. Then, for each $j\in \bN, I_j ={( DX_1,DX_2)}^jA$.
\end{thm}

\begin{proof}  By (ii) of Lemma \ref{Wang}, there exist $f_1,f_2$ in $R$ such that the following hold.
\begin{itemize}
	\item $DX_1=f_1$ and $DX_2=f_2$,
	\item $A=R[f_2X_1-f_1X_2]~(=R^{[1]})$.
\end{itemize}

\smallskip\noindent
First, we will show that $I_1=(f_1,f_2)A$. It is clear that $(f_1,f_2)A\subseteq I_1.$ For the converse, let $h\in B$ be such that $Dh\in A$. Let $u:=f_2X_1-f_1X_2$. Since $D$ is irreducible, $f_1,f_2$ are mutually coprime and hence $u$ is irreducible in $B$. Let $Dh=P(u)\in R[u]$. Since $D(f_1h-X_1P(u))=0=D(f_2h-X_2P(u))$, we have $$f_1h=X_1P(u)+Q_1(u) \text{ and } f_2h=X_2P(u)+Q_2(u) \text{ for some }Q_1(u),Q_2(u)\in R[u].$$ Therefore, $uh=(f_2h)X_1-(f_1h)X_2=X_1Q_2(u)-X_2Q_1(u)$.

\smallskip\noindent
Let $c_1,c_2\in R$ and $P_1(u),P_2(u)\in R[u]$ be such that $Q_i(u)=c_i+uP_i(u)$ for $i=1,2.$ Then $uh=u(X_1P_2(u)-X_2P_1(u))+(c_2X_1-c_1X_2)$ and hence $u\mid(c_2X_1-c_1X_2)$. 
Let $c_2X_1-c_1X_2=cu$ for some $c\in B$. Then $h=X_1P_2(u)-X_2P_1(u)+c$. It is clear that $c\in R$ and hence $Dh=P_2(u)f_1-P_1(u)f_2\in(f_1,f_2)A$.

\medskip\noindent
Let $C=R[X_1,X_2,Z,S_1,S_2]~(=R^{[5]})$ and $\lambda$ a weighted degree map on $C$ defined by $\lambda(X_1)=\lambda(X_2)=\lambda(S_1)=\lambda(S_2)=1$, $\lambda(Z)=0$. Suppose $J:=(Z-u,S_1-X_1,S_2-X_2)C$. By Proposition \ref{top degree ideal}, $\widetilde{J}=(u,S_1-X_1,S_2-X_2)$. Since $\widetilde{J}$ is a prime ideal of $C$, by Corollary \ref{primality}
$I_j=\underset{i_1+i_2=j}\sum D ^j(X_1^{i_1}X_2^{i_2})A$. Hence, by Lemma \ref{prule}, $I_j=\underset{i_1+i_2=j}\sum {f_1}^{i_1}{f_2}^{i_2}A$. Thus, $I_j={(f_1,f_2)}^{j}A$.
\end{proof}  

\smallskip
\noindent
The following results are immediate corollaries of Theorem \ref{BhD} and Theorem \ref{inice}.

\begin{cor}\label{ufd2}
Let $R$ be a {\rm UFD} and $B=R[X,Y]$. Let $D$ be an irreducible $R$-lnd on $B$ such that $D^2X=D^2Y=0$. Let $A={\rm Ker}(D)$. Then, for each $n \geq 1 $ the following hold.
\begin{enumerate}
	\item [\rm(i)]If $D$ is fixed point free, then $I_n=A$;
	\item [\rm(ii)] If $D$ is not fixed point free, then $I_n$ is generated by $n+1$ elements of $R$ with ${\rm grade}_A(I_n)=2${\small{\small \footnote{for a Noetherian ring $R$ and an ideal $I$ of $R$,  ${\rm grade}_A(I)$ denotes the length of the maximal $R$-regular sequence contained in $I$.}}}, and hence $I_n$ is not principal. 
\end{enumerate}
\end{cor}

\begin{cor}\label{pid}
Let $R$ be a {\rm PID} and $B=R[X,Y,Z]$. Let $D$ be an irreducible $R$-lnd on $B$ such that $D^2X=D^2Y=D^2Z=0$. Let $A={\rm Ker}(D)$. Then, for each $n \geqslant 1 $ the following hold.
\begin{enumerate}
	\item [\rm(i)]If $D$ is fixed point free, then $D$ has a slice and hence $I_n=A$; 
	\item[\rm(ii)]If $D$ is not fixed point free, then $I_n$ is generated by $n+1$ elements with ${\rm grade}_A(I_n)=2$, and hence $I_n$ is not principal.
\end{enumerate}
\end{cor}

\begin{proof} The proof follows from Theorem \ref{Nice} and Theorem \ref{inice}.
\end{proof}

\begin{rem}
\em{Part (i) of Corollary \ref{pid} holds even if  $R$ is a  Dedekind domain. In fact, in this case too $D$ has a slice (see \cite[Proposition 3.8]{DG}).}
		

\end{rem}

\section{Structure of $I_n~(n\geqslant 1)$ for strictly $1$-quasi-nice derivations}
We now turn to the case of strictly $1$-quasi-nice derivations on $R^{[2]}$, where $R$ is a UFD. If the derivation is fixed point free, it follows from Corollary \ref{2varufd} that each image ideal is the whole ring. The following result describes the plinth ideal of a strictly $1$-quasi-nice derivation $D$ on $R[X_1,X_2]$, which is not fixed point free and $D^2X_1=0$ with $DX_1$ irreducible.  
\begin{prop}\label{plinth_quasi} Let $R$ be a {\rm UFD} and $B=R[X_1,X_2]$. Let $D$ be an irreducible $R$-lnd on $B$ which is not fixed point free. Suppose $A={\rm Ker}(D)$. If $D^2X_1=0$ and $DX_1$ is irreducible, then the following are equivalent.
\begin{enumerate}
	\item [\rm (I)] $I_1$ is principal.
	\item[\rm (II)] $I_1=(DX_1)A$.
	\item [\rm (III)] $D$ is strictly $1$-quasi-nice.
	
\end{enumerate}
\end{prop}

\begin{proof} Since (II)$\implies$(I) is obvious, it is enough to show that (I)$\implies$(III) and (III)$\implies$(II).

\noindent\medskip
(I) $\Rightarrow$ (III): Suppose that $D$ is a nice derivation. Then, there exists a coordinate system $\{U,V\}$ in $B$ such that $D^2U=D^2V=0$.  By Lemma \ref{Wang}(ii), there exist $p,q \in R$, such that $DU=p$ and $DV=q$. Since $D$ is not fixed point free, by Theorem \ref{BhD} $p,q$ form a regular sequence in $B$ and hence in $A$. So $I_1$ is not principal.

\medskip
\noindent
(III) $\Rightarrow$ (II): Assume (III) holds. By Lemma \ref{Wang}(i), there exist $b\in R$ and $f(X_1)\in{R[X_1]}$ such that the following hold:
\begin{itemize}
	\item $DX_1=b$ and $DX_2=-f^{\prime}(X_1)~(=\dfrac{df}{dX_1})$.
	\item $A=R[bX_2+f(X_1)]~(=R^{[1]})$.
\end{itemize}

\noindent
Without loss of generality, we can assume that $f$ has no constant term. Since $D$ is not fixed point free, it follows from Theorem \ref{BhD} that $b,f^{\prime}$ form a regular sequence in $B$ and hence $b,f$ is a regular sequence in $B$.  So, there exists $~u(X_1),~v(X_1)$ in $R[X_1]$ such that $f(X_1)=u(X_1)+bv(X_1)$, none of the coefficients of $u(X_1)$ is divisible by $b$ and ${\rm deg}_{X_1} u(X_1)\geqslant 1$. Let $\sigma:B\longrightarrow{B/bB}$ be the natural projection map and set $\bar{A}:=\sigma(A), \bar{R}:=\sigma(R)$. Then $\bar{A}=\bar{R}[\sigma(u)]$. 

\smallskip
\noindent We will show that $I_1 \subseteq bA$. Otherwise, there exists $h\in{B}$ be such that $Dh\in{A\setminus bA}$. Since $bB\cap A=bA$, $\sigma(Dh)\neq 0$. Let $\sigma(Dh)=P(\sigma(u))\in{\bar{R}[\sigma(u)]}$. Since $X_1Dh-hDX_1\in{A}$, there exists $0\neq Q(\sigma(u))$ such that $\sigma(X_1)P(\sigma(u))=Q(\sigma(u))$. Let $K$ be the field of fractions of the integral domain $R/bR$. Then, $K(\sigma(X_1))=K(\sigma(u))$. Then ${\rm deg}_{\sigma(X_1)}(\sigma(u))=1$ and hence ${\rm deg}_{X_1}(u)=1$. Let $u(X_1)=aX_1+c$ for some $a,c\in R$. Setting  $X_1^\prime=X_1$ and $X_2^\prime=X_2+v(X_1)$ we get $D^2X_1^\prime=D^2X_2^\prime=0$, contradicting the fact that $D$ is strictly $1$-quasi-nice.
\end{proof}

Now, we describe the higher image ideals of $D$ under the hypotheses of Proposition \ref{plinth_quasi}.
\begin{thm}\label{2varquasi}
Let $R$ be a {\rm UFD} and $B=R[X_1,X_2]$. Let $D$ be an irreducible $R$-lnd on $B$, which is not fixed point free. Let $A={\rm Ker}(D)$. Suppose that the following conditions hold.
\begin{enumerate}
	\item[\rm(i)]$DX_1~(=b)\in R$ is irreducible,
	\item[\rm(ii)]$DX_2=-f^{\prime}(X_1)$ where $f(X_1)=\overset{m}{\underset{i=0}{\sum}} f_i{X_1}^i$, $f_m\neq 0$,
	\item[\rm(iii)]$D$ is strictly $1$-quasi-nice. 
\end{enumerate}
Let $d={\rm Max}\{i\in\bN:i\leqslant m \text{ and } b\nmid f_i\}$
Then, $d\geqslant 2$ and for each $j\geqslant 2$, $I_j=b^{m_j}\bar{I_j}$, where $m_j={\rm Min}\{i_1+(d-1)i_2:i_1,i_2\in\bN,i_1+di_2=j\}$, $\bar{I_j}={(b,f_d)}^{l_j}A$ and $l_j=[j/d].$

\end{thm}
\begin{proof}

\noindent
It is clear that we can assume $f_0=0$. Since $D$ is irreducible, $\{i\in\bN:i\leqslant m \text{ and } b\nmid f_i\}$ is a non-empty set. Let $f(X_1)=u(X_1)+bv(X_1)$, where none of the coefficients of ${X_1}^i$ in $u(X_1)$ is divisible by $b$ . Set $Y_1=X_1$ and $Y_2=X_2+v(X_1)$. Then $DY_1=b$ and $DY_2=-u^{\prime}(Y_1)$. It is clear that $f_d$ is the leading coefficient of $u(Y_1)$. Since $D$ is not nice, $d\geqslant 2$.


\smallskip\noindent   
Let $j\geqslant 2$, $C=R[Y_1,Y_2,Z,S]$ and $\lambda$ a weighted degree map on $C$ defined by $\lambda(Y_1)=1,\lambda(Y_2)=d,\lambda(S)=1$ and $\lambda(Z)=0$. Suppose $J:=(Z-bY_2-u(Y_1),S-Y_1)C$. By Proposition \ref{top degree ideal}, $\widetilde{J}=(bY_2+f_dY_1^d,S-Y_1)$. Since ${\rm gcd}(b,f_d)=1$, $\widetilde{J}$ is a prime ideal of $C$ and hence by Corollary \ref{primality}, we have 
$I_j=\underset{i_1+di_2=j}\sum D ^j(Y_1^{i_1}Y_2^{i_2}) A$ and hence by Lemma \ref{prule}, $I_j=\underset{i_1+di_2=j}\sum {(DY_1)}^{i_1}{(D^dY_2)}^{i_2}A$. We note that $D^dY_2=D^{d-1}(-u^{\prime}(Y_1))=-d!b^{d-1}f_d$. Thus, $I_j=\underset{i_1+di_2=j}\sum b^{i_1+(d-1)i_2}{f_d}^{i_2}A=b^{m_j}\bar{I_j}$, where $\bar{I_j}=\underset{i_1+di_2=j}\sum b^{i_1+(d-1)i_2-m_j}{f_d}^{i_2}A$. Let $S_j=\{(i_1,i_2):i_1,i_2\in\bN, i_1+di_2=j\}$. Since for any $(i_1,i_2)\in S_j$, $i_1+(d-1)i_2=j-i_2$, we have $m_j=j-l_j$. Hence $\bar{I_j}=\overset{l_j}{\underset{i_2=0}\sum} b^{l_j-i_2}{f_d}^{i_2}A$.
\end{proof}

%
%

\begin{cor}

In {\rm Theorem \ref{2varquasi}}, if $ (b,f_d)A=A$, then $I_j=b^{m_j}A$ for each $j\geqslant 1$.

\end{cor}

\begin{cor}
In {\rm Theorem \ref{2varquasi}}, if $1\leqslant j<d$, then $I_j=b^{j}A$.
\end{cor}

\begin{proof}
If $1\leqslant j<d$, then $l_j=0$ and hence $m_j=j$.  
\end{proof}
\begin{rem}
\em{Let $R$ be a PID, $D$ an irreducible and not fixed point free $R$-lnd on $R[X_1,X_2]$. Then, by Lemma \ref{Wang}(iii), $D$ is quasi-nice if and only if it is strictly $1$-quasi-nice.}
\end{rem}

\indent
The next proposition shows that we can remove the condition ``$DX_1$ is irreducible'' from Proposition \ref{plinth_quasi} when $R$ is a DVR.
\begin{prop}\label{plinth_quasi_DVR} Let $(R,p)$ be a {\rm DVR} with parameter $p$ and $B=R[X_1,X_2]$. Let $D$ be an irreducible $R$-lnd on $B$, which is not fixed point free.  Suppose $A={\rm Ker}(D)$. If $DX_1 \in R$ and $DX_2=-f^{\prime}(X_1)$, where $f(X_1)$ is a polynomial in $X_1$ with degree $d$ and leading coefficient $f_d$, then we have the following.

\begin{enumerate}
	\item [\rm(i)]	$I_1=(DX_1)A$.
	\item [\rm(ii)] If $p\nmid f_d$, then for each $j\geqslant 1$, $I_j={(DX_1)}^{m_j}A$, where $m_j={\rm Min}\{i_1+(d-1)i_2:i_1,i_2\in\bN,i_1+di_2=j\}$.
\end{enumerate}

\end{prop}

\begin{proof} Let $DX_1=b$. By Lemma \ref{Wang}(i), $A=R[F]~(=R^{[1]})$, where $F=bX_2+f(X_1)$.

\noindent
Without loss of generality, we can assume that $f$ has no constant term. Since $D$ is not fixed point free, we can assume that $b=p^n$ with $n\geqslant 1$. It follows from Theorem \ref{BhD} that $b,f^{\prime}$ form a regular sequence in $B$ and hence $b,f$ form a regular sequence in $B$. Then $p,f$ also form a regular sequence in $B$. So, there exist $u(X_1),~v(X_1)$ in $R[X_1]$ such that $f(X_1)=u(X_1)+pv(X_1)$, none of the coefficients of $u(X_1)$ is divisible by $p$ and ${\rm deg}_{X_1} u(X_1)\geqslant 1$. Let $\sigma:B\longrightarrow{B/pB}$ be the natural projection map and set $\bar{A}:=\sigma(A), \bar{R}:=\sigma(R)$. Then $\bar{A}=\bar{R}[\sigma(u)]$. 

\medskip\noindent
(i): Let $h\in B$ be such that $Dh\in I_1$. Then $Dh={p}^{m}g$ where $g\in A\setminus pA$ and $m\geqslant 0$. If $m\geqslant n$, then $Dh\in bA$. If possible, suppose that $m<n$. Since $X_1Dh-bh\in A$, we have $X_1g-p^{n-m}h\in A$. Let $\sigma(g)=P(\sigma(u))\in{\bar{R}[\sigma(u)]}$. Since $\sigma(g)\neq 0$, there exists $ Q(\sigma(u))~(\neq 0)$ such that $\sigma(X_1)P(\sigma(u))=Q(\sigma(u))$.  Then, $\frac{R}{pR}(\sigma(X_1))=\frac{R}{pR}(\sigma(u))$ and hence $\frac{R}{pR}[\sigma(X_1)]=\frac{R}{pR}[\sigma(u)]$.  So, $B\otimes_R\frac{R}{pR}=\frac{R}{pR}[\sigma(F)]^{[1]}$, where $F=bX_2+f(X_1)$. Again, $B\otimes_RR[\frac{1}{p}]=R[\frac{1}{p}][F,X_1]$.  Hence, $F$ is a residual coordinate in $R[X_1,X_2]$ and hence a coordinate in $R[X_1,X_2]$ (see \cite[Theorem 3.2]{BD1}).
So, $D$ has a slice in $B$, contradicting the fact that $D$ is not fixed point free. 

\medskip\noindent
(ii): Let $C=R[X_1,X_2,Z,S]$ and $\lambda$ a weighted degree map on $C$ defined by $\lambda(X_1)=1,\lambda(X_2)=d,\lambda(S)=1$ and $\lambda(Z)=0$. Suppose $J:=(Z-p^nX_2-f(X_1),S-X_1)C$. By Proposition \ref{top degree ideal}, $\widetilde{J}=(p^nX_2+f_dX_1^d,S-X_1)$. Since $f_d\in R^{*}$, $\widetilde{J}$ is a prime ideal of $C$ and hence by Corollary \ref{primality}, we have $I_j=\underset{i_1+di_2=j}\sum D ^j(X_1^{i_1}X_2^{i_2}) A$. Following the same arguments given in Theorem \ref{2varquasi}, we have for each $j\geqslant 2$, $I_j=b^{m_j}\bar{I_j}$, where $m_j={\rm Min}\{i_1+(d-1)i_2:i_1,i_2\in\bN,i_1+di_2=j\}$, $\bar{I_j}={(b,f_d)}^{l_j}A$ and $l_j=[j/d].$ Since $(b,f_d)R=R$, $\bar{I_j}=A$ and hence $I_j=b^{m_j}A$.
\end{proof}

\begin{rem}\label{degree one}
{\em In Proposition \ref{plinth_quasi_DVR}, if $DX_1=p$, then condition $p\nmid f_d$ can be obtained by making a suitable coordinate change and using the fact that $D$ is not nice.
}
\end{rem}

Theorem \ref{plinth_quasi_PID} shows that Proposition \ref{plinth_quasi_DVR} can be extended to a PID $R$. First, we prove an elementary lemma.

\begin{lem}\label{local ff}
Let $R$ be a {\rm PID} containing $\bQ$, $B=R[X_1,X_2]$ and $D$ be an $R$-lnd on $B$ satisfying the following conditions.

\begin{enumerate}
	\item [\rm(i)] $DX_1=p^n$ for some prime element $p$ of $R$ and a positive integer $n$,
	\item [\rm(ii)] $DX_2=-f^{\prime}(X_1)$, where $f(X_1)=\overset{d}{\underset{i=0}{\sum}} f_iX^i, f_d\neq 0$.
\end{enumerate} 
Then, $D$ is fixed point free if and only if $p\mid f_j$ for all $j\geqslant 2$ and $p\nmid f_1$.
\end{lem}

\begin{proof}
Let $D:B\longrightarrow B$ be an $R$-lnd satisfying (i) and (ii). Then, $D$ is fixed point free if and only if $(p^n, f^{\prime}(X_1))B=B$. It is easy to observe that $(p^n, f^{\prime}(X_1))B=B$ if and only if $p\mid f_j$ for all $j\geqslant 2$ and $p\nmid f_1$.
\end{proof}

\begin{thm}\label{plinth_quasi_PID}
Let $R$ be a {\rm PID}, $B=R[X_1,X_2]$ and $D$ an irreducible $R$-lnd on $B$ satisfying the following conditions.

\begin{enumerate}
	\item [\rm(i)] $DX_1=\underset{i\in I}\prod{p_i}^{r_i}~(\in R)$, where $I=\{1,2,\dots,n\}$ and for each $i\in I$, $p_i$ is a prime element in $R$,
	\item[\rm(ii)] $DX_2=-f^{\prime}(X_1)$, where $f(X_1)=\overset{d}{\underset{j=0}{\sum}} f_jX^j, f_j\in R~(0\leqslant j\leqslant d) \text{ and }f_d\neq 0$.
\end{enumerate}
Let $A={\rm Ker}(D)$ and  $J=\{i\in I:p_i\mid f_j \text{ for all }j=2,\dots,d \text{ and }p_i\nmid f_1\}$. Then following hold.

\begin{enumerate}
	\item[\rm(I)] $I_1=
	\begin{cases}
		({\underset{i\in I\setminus J}\prod}{p_i}^{r_i})A, & \text{if } J\neq I;\\
		A, & \text{if } J=I.
	\end{cases}
	$
	\item[\rm(II)]  If for each $i\in I$, $p_i\nmid f_d$, then for each integer $j\geqslant 1$, $I_j=(DX_1)^{m_j}A$, where $m_j={\rm Min}\{i_1+(d-1)i_2:i_1,i_2\in\bN,i_1+di_2=j\}$.
	
	
\end{enumerate}
\end{thm}

\begin{proof} 
Let $D_{p_i}$ be the natural extension of $D$ to the $R_{(p_i)}$-lnd on $R_{(p_i)}[X_1,X_2]$. Suppose ${I_j}({i})$ denote the $j$-th higher image ideal of $D_{p_i}$ for $1\leqslant i\leqslant n$ and $j\geqslant 1$. Then $I_j(i)=I_jA_{(p_i)}$, where $A_{(p_i)}$ is the localisation of $A$ under the multiplicatively closed set $R\setminus p_iR$.

\medskip\noindent
Let $j$ be a positive integer. By Lemma \ref{local ff}, for each $i\in I\setminus J$ the $R_{(p_i)}$-lnd $D_{p_i}$ is not fixed point free and hence by Proposition \ref{plinth_quasi_DVR}, for each $i\in I\setminus J$, $I_jA_{(p_i)}=({p_i}^{r_i})^{m_j}A_{(p_i)}$, where $m_j={\rm Min}\{i_1+(d-1)i_2:i_1,i_2\in\bN,i_1+di_2=j\}$; and $I_jA_{(p_i)}=A_{(p_i)}$ if $i\in J$. Let $\m$ be a maximal ideal of $A$ and $\p=\m\cap R$. 
If $p_i\notin\p$ for all $i$, then $DX_1$ is a unit in $A_{\m}$ and hence $I_jA_\m=A_\m$ by Lemma \ref{ii_slice}. 
If there exists $i$ such that $p_i\in \p$, then $\p=p_iR$ and hence $I_jA_\m~(=(I_jA_\p)A_\m)=
\begin{cases}  
	({p_i}^{r_i})^{m_j}A_\m, & \text{if } i\in I\setminus J;\\
	A_\m, & \text{if } i\in J. 
\end{cases}
$\\
\noindent {\bf Case-1:} $J=I$.
Since $I_j={\underset{\m\in {\rm maxSpec}(A)}\bigcap} I_jA_\m$, we have $I_j=\underset{i\in I}\bigcap A_\m=A$. 

\noindent {\bf Case-2:} $J\neq I$. Then $I_j={\underset{\m\in {\rm maxSpec}(A)}\bigcap} I_jA_\m= {\underset{i\in I\setminus J}\bigcap} ({p_i}^{r_i})^{m_j}A_\m$.
Since $I_j\subseteq A$ and $({p_i}^{r_i})^{m_j}A_\m\cap A=({p_i}^{r_i})^{m_j}A$ for each $i\in I$, we have $I_j={\underset{i\in I\setminus J}\bigcap}({p_i}^{r_i})^{m_j}A$. Since $A$ is a UFD, $I_j=({\underset{i\in I\setminus J}\prod}{p_i}^{r_i})^{m_j}A$.

\medskip\noindent
Now, (I) can be obtained by taking the positive integer $j$ to be $1$. (II) follows from the fact that if for each $i\in I$, $p_i\nmid f_d$, then $J=\emptyset$.
\end{proof}

\begin{rem} \em{We observe the following facts.
	
	\begin{enumerate}
		\item In Theorem \ref{plinth_quasi_PID},	if $r_i=1$ for all $i\in I$ then for each integer $j\geqslant 1$, $I_j=({\underset{i\in I\setminus J}\prod}{p_i})^{m_j}A$, where $m_j={\rm Min}\{i_1+(d-1)i_2:i_1,i_2\in\bN,i_1+di_2=j\}$.
		
		\begin{proof}
			Follows from Remark \ref{degree one} and part (II) of Theorem \ref{plinth_quasi_PID}.
		\end{proof}
		
		\item In the setup of Theorem \ref{plinth_quasi_PID}, if  for each $i\in I$ the induced $R_{(p_i)}$-lnd $D_{p_i}$ is fixed point free (i.e., $J=I$), then $D$ is fixed point free.
	\end{enumerate}}

\end{rem}

 \noindent The following example shows that in Proposition \ref{plinth_quasi_DVR}(ii), if we remove the condition $p\nmid f_d$, then $I_j$ may not be generated by $(DX_1)^{m_j}$.
	
	\begin{ex}\label{ex}
{\em
	Let $R=k[T]_{(T)}$, $B=R[X_1,X_2]$ and $D$ be an $R$-lnd on $B$ defined by $DX_1=T^3$ and $DX_2=-f^{\prime}(X_1)$ where $f(X_1)=T^{2}X_1^4+TX_1^{3}+X_1^2+TX_1$. Then $D$ is an irreducible $R$-lnd which is not fixed point free. Let $F=T^3X_2+f(X_1)$. Then $A:={\rm Ker}(D)=R[F]$. 
	
	\smallskip\noindent 
Let $S_j=\{(i_1,i_2):i_1,i_2\in\bN, i_1+4i_2=j\}$.	Then $S_3=\{(3,0)\}$ and $m_3={\rm Min}\{i_1+3i_2:i_1,i_2\in\bN,i_1+4i_2=3\}=3$. It is easy to see that $D^{3}(TX_2+X_1^4)=-6T^8$. So $T^8\in I_{3}$ and hence $I_3$ can't be generated by $(DX_1)^{m_3}=T^9$.
	
}
\end{ex}

\noindent With Example \ref{ex} in mind, we pose the following question.

\smallskip
\begin{ques}\label{Q1}
\it{Is $I_3= T^8A?$}
\end{ques}

\begin{center}
{\bf Acknowledgement}
\end{center}
\noindent
This research is supported by the Indo-Russia Project DST/INT/RUS/RSF/P-48/2021 with TPN 64842.

{\small{
	
	}}


\begin{thebibliography}{99999}
		\bibitem{A} B. Alhajjar, {\it LND-Filtrations and Semi-Rigid Domains,} (2015)\\ arXiv. https://doi.org/10.48550/arXiv.1501.00445.
		
		\bibitem{KKO22} A. Ben Khaddah, M. El Kahoui and M. Ouali, {\it The freeness property for locally nilpotent derivations of $R^{[2]}$}, Transformation Groups {\bf 28} (2023) 35--60. 
		
		\bibitem{BD1} S. M. Bhatwadekar and A. K. Dutta, 
		{\it On residual variables and stably polynomial algebras,}
		Comm. Algebra {\bf 21(2)} (1993) 635--645.
		
		\bibitem{BD97} S.M. Bhatwadekar, A.K. Dutta,  {\it Kernel of locally nilpotent R-derivations of $R[X,Y]$,}   Trans. Amer. Math. Soc. {\bf 349(8)} (1997) 3303--3319.
		
		\bibitem{DF} D. Daigle, G. Freudenburg,
		{\it Locally nilpotent derivatons over a UFD and an application to rank two locally nilpotent derivations 
			of $k[X_1,\dots,X_n]$}, 
		J. Algebra {\bf 204} (1998) 353-371. 
		
		\bibitem{DK} D. Daigle, S. Kaliman, {\it A note on locally nilpotent derivations and variables of $k[x,y,z]$}, Canad. Math. Bull. {\bf 52} (2009) 535--543.
		
		
		\bibitem{DG} N. Dasgupta, N. Gupta,
		{\it Nice derivations over principal ideal domains,}
		J. Pure Appl. Algebra,
		{\bf 222(12)} (2018) 4161--4172.
		
		
		\bibitem{DeF} J. K. Deveney, D. R. Finston,
		{\it A proper $\bG_a$-action on $\bC^5$ which is not locally trivial,} Proc. Amer. Math. Soc. {\bf 123}
		(1995), 651–-655.
		
		
		\bibitem{F} G. Freudenburg, {\it Algebraic Theory of Locally Nilpotent Derivations, Second Edition}, Springer-Verlag, Berlin, Heidelberg (2017).
		
		
		\bibitem{KML} S. Kaliman, L. Makar-Limanov, {\it $AK$-invariant of affine domains}, Affine Algebraic Geometry,
		Osaka University Press, Osaka, (2007) 231--255.
		\bibitem{R} R. Rentschler, {\it Op\'erations du groupe additif sur le plan affine}, C. R. Acad. Sc. Paris {\bf 267} (1968) 384--387.
		
		\bibitem{W} Z. Wang, 
		{\it Homogenization of locally nilpotent derivations and an application to $k[X,Y,Z]$}, 
		Journal of Pure and Applied Algebra {\bf 196} (2005) 323-337.
		
		\end{thebibliography}
\end{document}